\documentclass{amsart}
\usepackage{amsmath, amsthm, amsfonts, amssymb}
\usepackage{mathrsfs}
\usepackage{microtype}
\usepackage[utf8]{inputenc}
\providecommand{\noopsort[1]{}}
\usepackage{bbm}
\usepackage{dsfont}
\usepackage{ifthen}
\usepackage{nicefrac}
\numberwithin{equation}{section}
\usepackage{a4}
\usepackage[active]{srcltx}
\usepackage{verbatim}
\usepackage{hyperref}

\usepackage[shortlabels]{enumitem}
\setlist{leftmargin=*}
\setlist[1]{labelindent=1.2\parindent}

\newtheorem{thm}{Theorem}[section]
\newtheorem{cor}[thm]{Corollary}
\newtheorem{prop}[thm]{Proposition}
\newtheorem{lem}[thm]{Lemma}

\theoremstyle{remark}
\newtheorem{rem}[thm]{Remark}

\newtheorem{hyp}[thm]{Hypothesis}
\newtheorem{example}[thm]{Example}

\theoremstyle{definition}
\newtheorem{defn}[thm]{Definition}
\newcommand{\coloneqq}{\mathrel{\mathop:}=}

\renewcommand{\Re}{{\rm Re}\,}

\newcommand{\eps}{\varepsilon}

\newcommand{\one}{\mathds{1}}

\newcommand{\CR}{\mathds{R}}
\newcommand{\CC}{\mathds{C}}

\newcommand{\CN}{\mathds{N}}
\newcommand{\CZ}{\mathds{Z}}

\newcommand{\cB}{\mathscr{B}}
\newcommand{\cM}{\mathscr{M}}

\newcommand{\cL}{\mathscr{L}}
\newcommand{\cA}{\mathscr{A}}

\newcommand{\la}{\langle}
\newcommand{\ra}{\rangle}

\begin{document}
\title{Diffusion with nonlocal boundary conditions}
\author{Wolfgang Arendt}
\email{wolfgang.arendt@uni-ulm.de}
\address{Institute of Applied Analysis, Ulm University, 89069 Ulm, Germany}

\author{Stefan Kunkel}
\email{stefan.kunkel@uni-ulm.de}
\author{Markus Kunze}
\email{markus.kunze@uni-ulm.de}
\address{Graduiertenkolleg 1100, Ulm University, 89069 Ulm, Germany}
\thanks{SK and MK were supported by the \emph{Deutsche Forschungsgemeinschaft} in the framework of the 
DFG research training group 1100}

\begin{abstract}
We consider second order differential operators $A_\mu$ on a bounded, Dirichlet regular set $\Omega \subset \CR^d$, subject to the nonlocal boundary conditions
\[
u(z) = \int_\Omega u(x)\, \mu (z, dx)\quad \mbox{for } z \in \partial \Omega.
\]
Here the function $\mu : \partial\Omega \to \cM^+(\Omega)$ is $\sigma (\cM (\Omega), C_b(\Omega))$-continuous with $0\leq \mu(z,\Omega) \leq 1$ for all $z\in \partial \Omega$. 
Under suitable assumptions on the coefficients in $A_\mu$, we 
 prove that $A_\mu$ generates a holomorphic positive contraction semigroup $T_\mu$ 
on $L^\infty(\Omega)$. The semigroup $T_\mu$ is never strongly continuous, but it enjoys the strong Feller property in the sense that it consists of kernel operators and takes values in $C(\overline{\Omega})$.  We also prove
that $T_\mu$ is immediately compact and study the  asymptotic behavior of $T_\mu(t)$ as $t \to \infty$.
\end{abstract}


\maketitle

\section{Introduction}

W.\ Feller \cite{feller-semigroup, feller-diffusion} has studied one-dimensional diffusion processes  and their corresponding transition semigroups. In particular, he characterized the boundary conditions which must be satisfied by functions in the domain of the generator of the transition semigroup. These include besides the classical Dirichlet and Neumann boundary conditions also certain nonlocal boundary conditions.

Subsequently, A.\ Ventsel' \cite{wenzell} considered the corresponding problem for diffusion problems on a domain 
$\Omega \subset \CR^d$ with smooth boundary. He characterized boundary conditions which can potentially occur for the generator of a transition semigroup. Naturally, the converse question of proving that a second order elliptic operator (or more generally, an integro-differential operator) subject to certain (nonlocal) boundary conditions indeed generates a transition semigroup, has recieved a lot of attention, see the article by Galakhov and Skubachevski{\u\i} \cite{galaskub}, the book by Taira \cite{taira} and the references therein. We would like to point out that the interest in general -- also nonlocal -- boundary conditions is not only out of mathematical curiosity. In fact, nonlocal boundary conditions appear naturally in applications, e.g.\ in thermoelasticity \cite{day},  in climate control systems \cite{gjs} and in financial methematics \cite{gk02}.\medskip

In this article, we consider second order differential operators on a bounded open subset $\Omega$ of $\CR^d$, 
formally given by
\[
\cA u \coloneqq \sum_{ij=1}^d  a_{ij}D_iD_j u  + \sum_{j=1}^d b_jD_ju +c_0u.
\]
Below, we will define a realization $A_\mu$ of $\cA$ on $C(\overline{\Omega})$ subject to nonlocal boundary conditions of the form
\[
u(z) = \int_\Omega u(x)\, \mu (z, dx)
\]
for all $z \in \partial \Omega$, where $\mu : \partial \Omega \to \mathscr{M}^+(\Omega)$ is a measure-valued function with 
$0\leq \mu (z, \Omega) \leq 1$. Here, $\mathscr{M}(\Omega)$ denotes the space of all (complex) Borel measures on $\Omega$ and 
$\mathscr{M}^+(\Omega)$ refers to the cone of positive measures.

This boundary condition has a clear probabilistic interpretation. Whenever a diffusing particle reaches the boundary (say at the point $z \in \partial \Omega$), it immediately jumps back to the interior $\Omega$. The point it jumps to is chosen randomly, according to the distribution $\mu(z,\cdot)$. In the case where $\mu (z, \Omega) < 1$, the particle ``dies'' with probability $1-\mu (z,\Omega)$. This is a multidimensional version of what Feller called in \cite{feller-diffusion} an \emph{instantaneous return process}. Stochastic processes of this form were constructed by Grigorescu and Kang \cite{gk02, gk07} and Ben-Ari and Pinsky \cite{b-ap07,b-ap09}.\medskip

We will now make precise our assumptions on $\Omega$, the coefficients $a_{ij}, b_j$ and $c_0$ as well as the measures 
$\mu$. Unexplained terminology will be discussed in Section 3.

\begin{hyp}\label{hyp1}
Let $\Omega$ be a bounded, open, Dirichlet regular subset of $\CR^d$ and $a_{ij} \in C(\overline{\Omega})$, $b_j, c_0 \in L^\infty(\Omega)$ be real valued for $i,j=1, \ldots, d$. We assume that the coefficients $a_{ij}$ are symmetric (i.e.\ $a_{ij}= a_{ji}$ for all $i,j=1, \ldots, d$) and strictly elliptic in the sense that there exists a constant $\eta>0$ such that for all $\xi \in \CR^d$ we have
\begin{equation}
\label{eq.elliptic}
\sum_{i,j=1}^d a_{ij}(x) \xi_i\xi_j \geq \eta |\xi|^2
\end{equation}
almost everywhere. Moreover, we assume that $c_0 \leq 0$ almost everywhere.

Finally, we assume one of the following regularity conditions:
\begin{enumerate}
[(a)]
\item the coefficients $a_{ij}$ are Dini continuous for all $i,j= 1,\ldots, d$ or
\item $\Omega$ satisfies the uniform exterior cone condition.
\end{enumerate}
\end{hyp}
Concerning the measure $\mu$, we assume the following.
\begin{hyp}\label{hyp2}
We assume that $\mu: \partial \Omega \to \mathscr{M}^+(\Omega)$ is $\sigma (\mathscr{M}(\Omega), C_b(\Omega))$-continuous and satisfies $0\leq \mu (z, \Omega) \leq 1$ for all $z \in \partial \Omega$.
\end{hyp}

In Hypothesis \ref{hyp2} we let $C_b(\Omega) \coloneqq \{ f: \Omega \to \CC, f \mbox{ is bounded and continuous}\}$ and denote by
$\sigma (\mathscr{M}(\Omega), C_b(\Omega))$ the weak topology induced by $C_b(\Omega)$ on $\mathscr{M}(\Omega)$ via the natural duality.

Hypotheses \ref{hyp1} and \ref{hyp2}  will be assumed throughout without further mentioning.\smallskip

For the purpose of this paper we define the space $W(\Omega)$ by
\begin{equation}\label{eq.w}
	W(\Omega)\coloneqq\bigcap_{1<p<\infty}W^{2,p}_\mathrm{loc}(\Omega).
\end{equation}
The reason for this choice is that by elliptic regularity \cite[Lemma 9.16]{gt} we have $u\in W(\Omega)$ whenever $u\in W^{2,p}_\mathrm{loc}(\Omega)$ for some $1<p<\infty$ and $\mathscr A u\in L^\infty_\mathrm{loc}(\Omega)$.

We  define the realization $A_\mu$ of $\cA$ in $L^\infty(\Omega)$ subject to nonlocal boundary conditions as follows.
\begin{equation}\label{eq.amu}
\begin{aligned}
D(A_{\mu}) & \coloneqq \Big\{ u \in C(\overline{\Omega})\cap W(\Omega) :\cA u \in L^\infty(\Omega), \\
&\qquad\qquad u(z) = \int_\Omega u(x)\, 
\mu (z, dx)\,\, \forall\, z \in \partial\Omega\Big\}\\
A_{\mu} u & \coloneqq \cA u.
\end{aligned}
\end{equation}

We can now formulate the main result of this article.

\begin{thm}\label{t.main}
Assume Hypotheses \ref{hyp1} and \ref{hyp2} and define the operator $A_\mu$ by Equation \eqref{eq.amu}. Then the following hold true:
\begin{enumerate}[(a)]
\item $A_\mu$ is the generator of a holomorphic semigroup $T_\mu := (T_\mu(t))_{t > 0}$
on $L^\infty(\Omega)$.
\item The semigroup $T_\mu$  consists of positive and contractive operators. Moreover, it is strongly Feller in the sense that $T_\mu(t)$ is a kernel operator taking values in $C(\overline{\Omega})$ for $t>0$. In particular, $T_\mu$ leaves the space $C(\overline{\Omega})$ invariant.
\item $T_\mu(t)$ is compact for every $t>0$.
\item There exist a positive projection $P$ of finite rank and constants $\eps>0$ and $M\geq 1$ such that
\[
\|T_\mu (t) - P\| \leq Me^{-\eps t}
\]
for all $t > 0$.
\end{enumerate}
\end{thm}

Theorem \ref{t.main} extends the results in the existing literature in several aspects. First of all, in the quoted results 
\cite{galaskub, taira, b-ap07, b-ap09}, $\Omega$ is assumed to have smooth boundary. Here, we only assume that $\Omega$ is  Dirichlet regular or (if the $a_{ij}$ are merely continuous) that $\Omega$ satisfies the uniform exterior cone condition  (see Section 3). Both are fairly weak regularity assumptions which are satisfied for Lipschitz boundaries. It seems that Dirichlet regularity is the weakest regularity assumption possible for this problem. Indeed, in the case of Lipschitz continuous diffusion coefficients it is proved in  \cite{ab99} that a  second order differential operators with Dirichlet boundary conditions (which corresponds to the choice $\mu=0$ above) generates a strongly continuous semigroup on $C_0(\Omega)$
if and only if $\Omega$ is Dirichlet regular. We also note that our regularity assumptions on the coefficients are much weaker than in the 
articles quoted above; in the case of Dirichlet regular $\Omega$ we assume that the diffusion coefficients are Dini-continuous (see Section 3) so that in particular H\"older-continuous coefficients are possible. If we assume slightly more regular boundary, we can even allow general continuous diffusion coefficients.
\smallskip 

We also obtain more information about the semigroup generated by the operator $A_\mu$. First of all, 
we obtain a semigroup on all of $C(\overline{\Omega})$ (even $L^\infty(\Omega)$) 
rather than on a closed subspace $C_\mu(\overline{\Omega})$ thereof  as in \cite{galaskub}.  This comes at the cost that our semigroup $T_\mu$ is not strongly continuous. Having a semigroup on all of $C(\overline{\Omega})$ has important consequences. For example, it follows that the semigroup $T_\mu$ is given through \emph{transition probabilities}, see Section 5 and also \cite[Theorem 9.2.2]{t88}. It is thus possible to construct -- in a canonical way, cf.\ \cite[Theorem 4.1.1]{ek} -- a Markov process with transition semigroup $T_\mu$. This gives an analytic approach to immediate return processes. In \cite{b-ap09}, the authors worked the other way round. They constructed the stochastic process directly and then used the process to study the transition semigroup on the $L^p$-scale. Another consequence of having a semigroup on $C(\overline{\Omega})$ is that, via duality, we obtain an adjoint semigroup $T_\mu^*$ on $\mathscr{M}(\overline{\Omega})$, the space 
of all (complex) measures on the Borel $\sigma$-algebra on $\overline{\Omega}$. This semigroup is important in probability theory, as it can be used to compute distributions of the associated Markov process from the initial distribution of the process. 
It now follows from part (d) of Theorem \ref{t.main} that the distributions of the associated Markov process converge in the total variation norm for every initial distribution.\smallskip

Second, we obtain from our techniques that the semigroup $T_\mu$ is holomorphic. Besides being interesting in its own right, it is the holomorphy of the semigroup which allows us to work with generators even though the semigroups under consideration are not strongly continuous.\smallskip 

Third, we prove that the semigroup $T_\mu$ is immediately compact.
This result is of particular importance.
In \cite[Proposition 4.7]{as14} it is proved that if $\Omega$ satisfies the uniform exterior cone condition, the semigroup generated by the Dirichlet Laplacian on $C_0(\Omega)$ is immediately compact. The proof is based on the fact that the domain of the Dirichlet Laplacian is contained in a H\"older space which, in turn, is compactly embedded into $C_0(\Omega)$. However assuming merely Dirichlet regularity this strategy to prove compactness cannot work. Here, we pursue a different approach, based on the strong Feller property. This is, once again, possible because we obtain a semigroup on all of $L^\infty(\Omega)$.

Using our additional information about the semigroup allows us to deduce information about the asymptotic behavior of the semigroup. If $\Omega$ is connected, $c_0=0$ 
and every measure $\mu(z)$ is a probability measure, we show in Corollary \ref{c.asympt}
that the semigroup converges to an equilibrium.
\smallskip 

Let us briefly compare our boundary conditions with other results in the literature. There one finds different, partially weaker, assumptions on the measures $\mu(z)$
defining the boundary conditions. In particular, Galakhov and Skubachevski{\u\i} \cite{galaskub} allow jumps on the boundary, i.e.\ 
$\mu(z)$ may have mass on the boundary whereas we require that for every $\mu(z)$ the boundary $\partial \Omega$ is a null set. Gurevich \cite{gur08} allows also functions $z \mapsto \mu(z)$ with certain discontinuities. Finally, Peng and Li \cite{pl13} consider the situation where the particle having reached the boundary first stays there before jumping back to the interior.
\smallskip

This article is structured as follows. Sections 2 and 3 contain preliminary results.
In Section 4, we prove a generation result from which in particular part (a) of Theorem \ref{t.main} follows. Here we also establish half of part (b) of our main theorem.
In Section 5, we recall the basic definitions concerning kernel operators and the strong Feller property. We then show that 
the resolvent of $A_\mu$ consists of strong Feller operators and deduce the rest of (b) as well as the remaining parts (c) and (d) of Theorem \ref{t.main} from this.

\section{Holomorphic semigroups}

The semigroups we are studying in this article are not strongly continuous. As this is not a standard situation, we recall the relevant definitions and results in this preliminary section. The following definition is taken from \cite[Section 3.2]{abhn}

\begin{defn}
Let $X$ be a Banach space. A \emph{semigroup} is a strongly continuous mapping $T : (0,\infty) \to \cL (X)$ such that
\begin{enumerate}
[(a)]
\item $T(t+s) = T(t)T(s)$ for all $t,s >0$;
\item there exist constants $M >0$ and $\omega \in \CR$ such that $\|T(t)\|\leq Me^{\omega t}$ for all $t>0$;
\item if $T(t)x=0$ for all $t>0$ it follows that $x=0$.
\end{enumerate}
We say that $T$ is \emph{of type} $(M,\omega)$ to emphasize that (b) holds with these constants.
\end{defn}

If $T$ is a semigroup of type $(M,\omega)$, there exists a unique operator $A$ such that $(\omega, \infty) \subset \rho (A)$ and
\[
R(\lambda, A)x = \int_0^\infty e^{-\lambda t} T(t)x\, dt
\]
for all $x \in X$ and $\lambda > \omega$, see \cite[Equation (3.13)]{abhn}. The operator $A$ is called the \emph{generator} of $T$.

We now characterize when $T$ is a \emph{contraction semigroup}, i.e.\ of type $(1,0)$.

\begin{prop}\label{p.contr}
Let $T$ be a semigroup of type $(M,\omega)$ 
on the Banach space $X$ and $A$ be the generator of $T$. The following are equivalent.
\begin{enumerate}
[(i)] 
\item $\|T(t)\| \leq 1$ for all $t>0$;
\item $\|\lambda R(\lambda, A)\| \leq 1$ for all $\lambda >\omega$.
\end{enumerate}
\end{prop}

\begin{proof}
If $\|T(t)\| \leq 1$ for all $t>0$ then for $\lambda >0$ we have
\begin{align*}
\|\lambda R(\lambda, A)x\| & = \Big\| \int_0^\infty \lambda e^{-\lambda t}T(t)x\, dt \Big\| \leq \int_0^\infty\lambda e^{-\lambda t}\|x\|\, dt = \|x\|.
\end{align*}
For the converse, we note that (ii) implies
\[
\frac{1}{n!} \|\lambda^{n+1}R(\lambda, A)^{(n)}\| = \|\lambda^{n+1}R(\lambda, A)^{n+1}\| \leq 1
\]
for all $n \in \CN_0$. Now the Post--Widder inversion formula \cite[Theorem 1.7.7]{abhn} yields $\|T(t)\|\leq 1$ for all $t>0$.
\end{proof}

A semigroup $T$ is called \emph{holomorphic} if it has a holomorphic extension to a sector
\[
\Sigma_\theta \coloneqq \big\{ r e^{i\varphi} : r>0, |\varphi| < \theta\big\}
\]
for some angle $\theta \in (0, \frac{\pi}{2}]$ which is bounded on $\Sigma_\theta\cap \{z \in \CC : |z|\leq 1\}$, see
\cite[Definition 3.7.1]{abhn}. We note that in this case the semigroup law automatically also holds for the holomorphic extension.
We call the semigroup $T$ \emph{bounded holomorphic} if the holomorphic extension is additionally bounded on all of $\Sigma_\theta$.

An operator $A$ generates a holomorphic semigroup if and only if there exists a constant $c \in \CR$ such that $A-c$ generates a bounded holomorphic semigroup. The generators of (bounded) holomorphic semigroups can be characterized by the following \emph{holomorphic estimate}.

\begin{thm}\label{t.holest}
An operator $A$ on $X$ generates a (bounded) holomorphic semigroup if and only if there exists $\omega \in \CR$ ($\omega=0$)
such that $\{ \lambda \in \CC : \Re \lambda >\omega\} \subset \rho (A)$ and
\[
\sup_{\Re\lambda > \omega}\|\lambda R(\lambda, A)\| < \infty.
\]
\end{thm}

\begin{proof}
Corollaries 3.7.12 and 3.7.17 of \cite{abhn}.
\end{proof}

A semigroup $T$ is called \emph{exponentially stable} if it is of type $(M, -\eps)$ for some $\eps>0$. An operator $A$ generates an exponentially stable semigroup if and only if $A+\eps$ generates a semigroup of type $(M, 0)$ for some $M\geq 0$ and $\eps>0$. 
To prove that a holomorphic semigroup $T$ is expontially stable it suffices to prove that the spectrum of its generator $A$ is contained in the open left half-plane. More precisely, we have the following result.

\begin{prop}\label{p.expstable}
Let $A$ be the generator of a holomorphic semigroup $T$. The following are equivalent.
\begin{enumerate}
[(i)]
\item $T$ is exponentially stable.
\item $\Re\lambda < 0$ for all $\lambda \in \sigma (A)$.
\end{enumerate}
\end{prop}

\begin{proof}
Assume (ii) and let $\omega$ be as in Theorem \ref{t.holest}. It follows from \cite[Theorem 3.7.11]{abhn}, that $\|\lambda R(\lambda, A)\|$ is bounded on the set
$\omega + \Sigma_\theta$ for a suitable $\theta \in (\frac{\pi}{2}, \pi)$. For $\eps>0$ the set $K_\eps \coloneqq \{ \lambda : \Re\lambda \geq -\eps\} \setminus (\omega +\Sigma_\theta)$ is compact. By (ii) we have $K_\eps \subset \rho (A)$ for a suitable $\eps>0$.  It follows that $\sup_{\lambda \in K_\eps}\|\lambda R(\lambda, A)\|< \infty$. Alltogether, $\sup_{\Re \lambda \geq - \eps}\|\lambda R(\lambda, A)\| < \infty$. This implies that $A+\eps$ generates a bounded holomorphic semigroup, whence $A$ generates an exponentially stable semigroup.

The converse is obvious.
\end{proof}

In this article, we are concerned with semigroups on the Banach spaces $X=L^\infty(\Omega)$ or $X= C(\overline{\Omega})$
or $X= C_0(\Omega) \coloneqq \{ u \in C(\overline{\Omega}): u|_{\partial \Omega} =0 \}$ where $\Omega \subset \CR^d$ is open and bounded. These spaces are Banach lattices in their natural ordering (see \cite[BI, CI]{schreibmaschine} for more details). From the point of view of applications to Markov processes it is important to consider positive semigroups. We recall that a semigroup $T$ on a Banach lattice $X$ is called \emph{positive} if $T(t)$ is a positive operator for all $t>0$. The latter means that if $u\geq 0$ then also $T(t)u\geq 0$.

We now characterize generators of positive semigroups.

\begin{prop}
\label{p.possg}
Let $A$ be the generator of a semigroup $T$ on a Banach lattice $X$. The following are equivalent.
\begin{enumerate}
[(i)]
\item $T$ is positive.
\item There exists some $\lambda_0 \in \CR$ such that $(\lambda_0, \infty) \subset \rho (A)$ such that $R(\lambda, A) \geq 0$ for all
$\lambda > \lambda_0$.
\end{enumerate}
\end{prop}

\begin{proof}
Let $T$ be of type $(M,\omega)$.
If $T$ is positive, then for $u\geq 0$ also $e^{-\lambda t}T(t)u\geq 0$ for all $t>0$ and $\lambda > \omega$. 
As the positive cone is closed, 
$R(\lambda, A)u = \int_0^\infty e^{-\lambda t}T(t)u\, dt \geq 0$ for all $\lambda > \omega$. The converse follows from the 
Post--Widder inversion formula \cite[Theorem 1.7.7]{abhn}, as
\[
(-1)^n\frac{1}{n!} R(\lambda, A)^{(n)}u = R(\lambda, A)^{n+1}u \geq 0
\]
whenever $u\geq 0$.
\end{proof}

Basically the same argument yields the following result concerning domination of semigroups.

\begin{prop}\label{p.dom}
Let $S$ and $T$ be positive semigroups on a Banach lattice $X$ with generators $B$ and $A$ respectively. The following are equivalent.
\begin{enumerate}
[(i)]
\item $S(t)\leq T(t)$ for all $t>0$;
\item There exists some $\lambda_0$ in $\CR$ with $(\lambda_0,\infty) \subset \rho (A)\cap\rho (B)$ and 
$R(\lambda, B) \leq R(\lambda, A)$ for all $\lambda>\lambda_0$.
\end{enumerate}
\end{prop}

Now let $T$ be a positive semigroup on a Banach lattice $X$ and let $A$ be its generator. The \emph{spectral bound} $s(A)$ of 
$A$ is defined as
\[
s(A) \coloneqq \sup \{ \Re\lambda: \lambda \in \sigma (A)\}.
\]
In general, $-\infty \leq s(A) < \infty$. However, if $s(A)>-\infty$, then $s(A) \in \sigma (A)$ by \cite[Proposition 3.11.2]{abhn}.
We now obtain the following spectral criterion for exponential stability.

\begin{prop}\label{p.posstable}
Let $T$ be a positive, holomorphic semigroup on a Banach lattice $X$ with generator $A$. If $[0,\infty) \subset \rho (A)$, then 
$T$ is exponentially stable.
\end{prop}

\begin{proof}
Since $s(A) \in \sigma (A)$ when $s(A)>-\infty$, it follows from our assumption that $s(A)<0$. Now Proposition \ref{p.expstable}
yields exponential stability.
\end{proof}

\section{Some properties of elliptic operators}

We start by recalling the main definitions used in Hypothesis \ref{hyp1}.

A bounded, open set $\Omega\subset\CR^d$ is called \emph{Dirichlet-regular} (or \emph{Wiener-regular}), if for each $\varphi\in C(\partial\Omega)$ there exists a (necessarily unique) function $u\in C^2(\Omega)\cap C(\overline\Omega)$ such that $\Delta u=0$ and $u|_{\partial\Omega}=\varphi$. In other words, $\Omega$ is Dirichlet regular if and only if the classical Dirichlet problem is well-posed.

Dirichlet regularity is a very weak notion of regularity which by Wiener's famous result \cite[§2.9]{gt} can be characterized by a capacity condition.
Examples of Dirichlet regular domains include those with Lipschitz boundary (or more general those sets which satisfy the uniform  exterior cone condition), all open subsets of $\CR$ and all simply connected open subsets of $\CR^2$. We refer to \cite{dl} for the proof of these results and to \cite{ad08} for further information on the Dirichlet problem; the proof of the two-dimensional result can be found in \cite[Chapter 21]{con}

A function $g: \overline{\Omega} \to \CR$ is called \emph{Dini-continuous} if the modulus of continuity
\[
\omega_g(t) \coloneqq \sup_{|x-y| \leq t}|g(x)-g(y)|
\]
satisfies
\[
\int_0^1\frac{\omega_g(t)}{t}\, dt < \infty.
\]
We note that in particular every H\"older continuous function on $\overline{\Omega}$ is Dini continuous.\smallskip

We now start our discussion of the differential operator $\cA$ with the following complex version of the maximum principle from \cite[Lemma 4.2]{as14}. Recall the definition of $W(\Omega)$ from \eqref{eq.w}.

\begin{lem}[Complex maximum principle]
	\label{l.localmax}
	Let $B=B(x_0,r)\Subset\Omega$ be a ball with center $x_0$ and radius $r>0$. Let $u\in W(\Omega)$ be a complex-valued function such that $\mathscr A u\in C(B)$. Assume that $|u(x)|\leq|u(x_0)|$ for all $x\in B$. Then
	\[
		\Re[\overline{u(x_0)}\mathscr A u(x_0)]\leq 0.
	\]
\end{lem}

We note that a priori $\cA u$ is an element of $L^p_{\mathrm{loc}}(\Omega)$ for all $1< p < \infty$, since $u$ is merely assumed to be in $W(\Omega)$. The hypothesis $\cA u \in C(B)$ in Lemma \ref{l.localmax} means that $\cA u$ coincides almost everywhere on $B$ with a continuous function $g$. In the conclusion of the lemma, 
$\cA u(x_0)$ is to be understood as $g(x_0)$.\smallskip

We can now prove that certain harmonic functions attain their maximum on the boundary.

\begin{lem}
	\label{l.maxprinciple2}
	Let $\Re\lambda>0$, $M\geq0$. Let $u\in C(\overline\Omega)\cap W(\Omega)$ such that $\lambda u -\mathscr Au=0$.
	
	If $|u(x)|\leq M$ on $\partial\Omega$, then $|u(x)|\leq M$  for all $x \in \overline{\Omega}$. If $M>0$, then we even have
	$|u(x)|<M$ for all $x \in \Omega$.
\end{lem}

\begin{proof}
	Suppose that $|u|$ attains its global maximum at a point $x_0\in\Omega$ with $|u(x_0)|\neq 0$. By Lemma \ref{l.localmax} $\Re[\overline{u(x_0)} \mathscr A u(x_0)]\leq0$. Since $\lambda u=\mathscr Au$ it follows that
	\[
		\Re\lambda|u(x_0)|^2=\Re[\overline{u(x_0)} \mathscr A u(x_0)]\leq0
	\]
	and thus $|u(x_0)|\leq 0$. This implies that $u\equiv0$, a contradiction.
\end{proof}

It is remarkable (and important for the rest of this article) that the notion of Dirichlet regularity, even though originally phrased in terms of the Laplace operator, is also sufficient for the well-posedness of the Poisson equation with respect to the 
elliptic operator $\mathscr{A}$. 

More precisely, we have the following result, in which we choose $p=d$ to apply Aleksandrov's maximum principle. The proof is 
based on \cite{as14}, where the Poisson problem is studied. The proofs in \cite{as14} use classical results from \cite{gt} and in particular results due to Krylov \cite{k67}.

\begin{prop}
	\label{p.wellposed}
	Let $\lambda\in\CC$ with $\Re \lambda\geq0$ be given. 
	Then for each $f\in L^d(\Omega)$, $\varphi\in C(\partial\Omega)$ there exists a unique $u\in C(\overline\Omega)\cap W^{2,d}_\mathrm{loc}(\Omega)$ such that 
	\begin{equation}
		\label{eq.dirichlet}
		\begin{aligned}
			\lambda u-\mathscr A u&=f\\
			u|_{\partial\Omega}&=\varphi.
		\end{aligned}
	\end{equation}
	Moreover, if $\lambda\in\CR$, $\lambda\geq0$, $f\geq0$ on $\Omega$ and $\varphi\geq0$ on $\partial\Omega$, then $u\geq0$ on $\overline\Omega$. Finally, if $f \in L^\infty_{\mathrm{loc}}(\Omega)$ then $u \in W(\Omega)$ by elliptic regularity.
\end{prop}

\begin{proof}
In the proof we use the \emph{Poisson operator} $\mathscr{P}$ on $L^d(\Omega)\oplus C(\partial \Omega)$, defind by
\begin{align*}
	D(\mathscr{P})&:=\{(u,0):\ u\in W^{2,d}_\mathrm{loc}(\Omega)\cap C(\overline{\Omega}),\, \cA u\in L^d(\Omega)\}\\
	\mathscr{P}(u,0)&:=(\cA u,-u|_{\partial\Omega}).
\end{align*}
It follows from \cite[Corollary 3.4]{as14} and the remarks after that corollary,
 that the operator $\mathscr{P}$ is bijective. Replacing for $\lambda >0$ the operator
$\cA$ with $\cA - \lambda$, we see that also $(\lambda -\mathscr{P})$ is bijective for $\lambda >0$. It is a consequence 
of Aleksandrov's maximum principle \cite[Theorem 9.1]{gt} (see also \cite[Theorem A.1]{as14}) that the inverse $(\lambda - \mathscr{P})^{-1}$ 
is a positive operator. Thus $\mathscr{P}$ is resolvent positive in the sense of \cite[Definition 3.11.1]{abhn} and $s(A)<0$. It now follows from \cite[Proposition 3.11.2]{abhn} that $\lambda-\mathscr{P}$ is invertible also for \emph{complex} $\lambda$ with
$\Re\lambda\geq0$. From \cite[Lemma 9.16]{gt} we infer that $u \in W(\Omega)$ whenever $u \in W^{2,p}_\mathrm{loc}(\Omega)$
for some $1<p<\infty$ and $\cA u  \in L^\infty_{\mathrm{loc}}$. This proves the last assertion.
\end{proof}

We note the following result on interior regularity.

\begin{prop}
	\label{p.local-regularity}
	Let $\Omega$ be Dirichlet regular, $U\Subset\Omega$ and $\Re \lambda \geq 0$. 
	Then there exists a constant $C=C(U) \geq0$ such that for all $f \in L^\infty(\Omega)$ and $\varphi \in C(\partial \Omega)$
	the solution $u$ of \eqref{eq.dirichlet} satisfies the estimate
	\[
		\|u\|_{C^1(\overline{U})}\leq C(\|f\|_{L^\infty(\Omega)}+\|\varphi\|_{C(\partial\Omega)}).
	\]
\end{prop}
\begin{proof}
As $f \in L^\infty(\Omega)$ we have $u \in W(\Omega)$ by elliptic regularity. By Sobolev embedding (\cite[Corollary 7.11]{gt})
$W(\Omega)\subset C^1(\Omega) $. Now the claim follows immediately from the closed graph theorem.
\end{proof}

If we consider $\mu \equiv 0$, i.e.\ $\mu (z) =0$ for all $z \in \partial \Omega$, then the operator $A_\mu$ is the realization $A_0$
of $\cA$ in $L^\infty(\Omega)$ with Dirichlet boundary conditions. In fact, $A_0$ is given by
\begin{equation}\label{eq.dbc}
\begin{aligned}
	D(A_0)&:=\{u\in C_0(\Omega)\cap W(\Omega):\ \cA u\in L^\infty (\Omega)\}\\
	A_0u&:=\cA u.
	\end{aligned}
\end{equation}
If $\Omega$ is Dirichlet regular, it follows from \cite[Theorem 4.1]{as14}, that the part of $A_0$ in $C_0(\Omega)$ generates a bounded holomorphic semigroup on $C_0(\Omega)$. Conversely, if the coefficients $a_{ij}$ are Lipschitz continuous, this characterizes Dirichlet 
regularity, see \cite[Theorem 4.10]{ab99}.

Concerning the operator $A_0$ on $L^\infty(\Omega)$, we have the following result.

\begin{thm}
	\label{t.sectorial}
	The operator $A_0$ generates an exponentially stable, positive, holomorphic semigroup $T_0$ on $L^\infty(\Omega)$.
	Moreover, $\|T_0(t)\| \leq 1$ for all $t >0$.
\end{thm}

\begin{proof}
It follows from Proposition \ref{p.wellposed} that $[0,\infty) \subset \rho (A_0)$ and $R(\lambda, A_0) \geq 0$ for all $\lambda \geq 0$.
We claim that $\|\lambda R(\lambda, A_0)\| \leq 1$ for all $\lambda >0$. To see this, put $u=R(\lambda, A_0)\one \geq 0$. As 
$R(\lambda, A_0)\geq 0$, we have $u\geq 0$ and it suffices to prove that $\lambda u \leq \one$. To this end,  
pick $x_0$ such that $u(x_0)= \max_{x \in \overline{\Omega}} u(x)$. If $u(x_0) = 0$, there is nothing to prove, 
so let us assume that $u(x_0)>0$, so that $x_0 \in \Omega$. As $\cA u = \lambda u - \one \in C(\overline{\Omega})$, it follows from Lemma \ref{l.localmax} that $\cA u(x_0) \leq 0$. Consequently, 
\[ \lambda R(\lambda , A_0)\one = \lambda u \leq \lambda u(x_0) 
= \lambda u(x_0) - 1 +1= \cA u(x_0) + 1 \leq 1,
\]
proving the claim.\smallskip

With this information at hand (which replaces \cite[Proposition 4.4]{as14}), the proof of \cite[Theorem 4.1]{as14} shows that $A_0$ generates a holomorphic semigroup $T_0$ on 
$L^\infty(\Omega)$. It follows from Proposition \ref{p.possg} that $T_0(t)\geq 0$ for $t>0$ and from Proposition \ref{p.contr} that 
$\|T_0(t)\| \leq 1$ for such $t$. Finally, Proposition \ref{p.expstable} implies that $T_0$ is exponentially stable.
\end{proof}

\section{Generation results}\label{sect.three}

In this section, we will prove that $A_{\mu}$ generates a holomorphic  semigroup $(T_\mu (t))_{t > 0}$ on $L^\infty (\Omega)$. We recall that
$A_\mu$ is defined by
\[
\begin{aligned}
D(A_{\mu}) & \coloneqq \Big\{ u \in C(\overline{\Omega})\cap W(\Omega) :\cA u \in L^\infty(\Omega), \\
&\qquad\qquad u(z) = \int_\Omega u(x)\, 
\mu (z, dx)\,\, \forall\, z \in \partial\Omega\Big\}\\
A_{\mu} u & \coloneqq \cA u.
\end{aligned}
\]

We first establish that $\lambda-A_{\mu}$ is injective for $\Re\lambda>0$.

\begin{lem}\label{l.amuinj}
Let $u \in D(A_{\mu})$ and $\Re\lambda >0$ be such that $\lambda u - \cA u =0$. Then $u=0$.
\end{lem}

\begin{proof}
By Lemma \ref{l.maxprinciple2}, there exists a point $z_0 \in \partial\Omega$ with 
$|u(z_0)| = \max\{ u(x) : x \in \overline{\Omega}\}$. Suppose that $u\neq 0$, i.e.\ $u(z_0) \neq 0$. We may assume without loss of generality that $|u(z_0)|=1$. Since $u(z_0) = \int_\Omega u(x) \, \mu (z_0, dx)$ we have $\mu (z_0) \neq 0$. Hence there exists 
a point $x_0 \in \Omega$ such that $\mu (z_0, B(x_0,r)) >0$ for all $r>0$ with $B(x_0,r) \subset \Omega$. 
Since $|u(x_0)| < 1$ by Lemma \ref{l.maxprinciple2}, we find $r>0$ and $\delta>0$ such that
$|u(x)| \leq 1-\delta$ for all $x \in B(x_0,r)$. Thus, with $\eps \coloneqq \mu (B(x_0,r))>0$, we obtain
\begin{align*}
1 & = |u(z_0)| \leq \int_\Omega | u(x)|\, \mu (z_0, dx)\\
& \leq \int_{B(x_0,r)} (1-\delta)\,\mu (z_0, dx) + \int_{\Omega\setminus B(x_0, r)} 1\, \mu (z_0, dx)\\
& = (1-\delta)\mu (z_0, B(x_0,r)) + \mu (z_0, \Omega\setminus B(x_0,r))\\
& = \mu (z_0, \Omega) - \delta \eps < 1.
\end{align*}
This is a contradiction.
\end{proof}

Thus, to prove that $A_\mu$ generates a holomorphic semigroup, it remains to show that $\lambda - A_\mu$ is surjective for $\Re\lambda >0$ and to establish the holomorphic estimate in Theorem \ref{t.holest}. We already know from Theorem \ref{t.sectorial} that $A_0$ generates a holomorphic semigroup. We will see that we can obtain the resolvent of $A_\mu$ from that of $A_0$ by a perturbation involving the operator $S_\lambda \in \cL (C(\overline{\Omega}))$ which is defined as follows.

 Given a function $v \in C(\overline{\Omega})$ we define 
$\varphi (z) \coloneqq \langle v, \mu (z)\rangle$. Since $\mu$ is $\sigma ( \mathscr{M}(\Omega), C_b(\Omega))$-continuous, 
we have $\varphi \in C(\partial \Omega)$. Thus, by Proposition \ref{p.wellposed}, there exists a unique function 
$u_\varphi$ with
\[
\lambda u_\varphi - \cA u_\varphi =0\quad\mbox{ and }\quad u_\varphi|_{\partial \Omega} = \varphi .
\]
We set $S_\lambda v \coloneqq u_\varphi$. Let us note that $S_\lambda v \in W(\Omega)$ for all $v \in C(\overline{\Omega})$
by elliptic regularity.

It follows from the complex maximum principle (Lemma \ref{l.maxprinciple2}) that $\|S_\lambda \|\leq 1$ for $\Re\lambda >0$. Moreover, $S_\lambda \geq 0$ whenever
$\lambda \in (0,\infty)$. We will prove that for $\Re\lambda >0$ the operator $(I-S_\lambda)$ is invertible and we have
\[
R(\lambda, A_\mu) = (I-S_\lambda)^{-1}R(\lambda, A_0).
\]
We start with 

\begin{prop}\label{p.S-properties}
For $\Re\lambda >0$ the operator $I-S_\lambda$ is invertible and for every $\delta>0$ we have
	\[
	 	\sup_{\Re\lambda\geq\delta}\|(I-S_\lambda)^{-1}\|<\infty.
	\] 
If $A_\mu$ is injective, then $I-S_0$ is invertible.
\end{prop}

The proof of Proposition \ref{p.S-properties} is given in a series of lemmas.

\begin{lem}\label{l.S-compact}
	$ S_\lambda^2 $ is compact for $\Re\lambda >0$ and for $\lambda =0$.
\end{lem}

\begin{proof}
	Let $u_n\in C(\overline{\Omega})$ be a bounded sequence. Recall that $S_\lambda u_n = u_{\varphi_n}$ where $\varphi_n(z)=\langle\mu(z),u_n\rangle$ and $u_{\varphi_n}$ solves the Dirichlet problem $\lambda u - \cA u =0$ with boundary values $\varphi_n$.
As a consequence of Proposition~\ref{p.local-regularity}  the sequence $ (u_{\varphi_n})_{n\in\CN} $ is locally equicontinuous on 
$\Omega$ so that, passing to a subsequence, we may and shall assume that $u_{\varphi_n}$ converges uniformly on each compact subset of $\Omega$ to a continuous function $u$. Note that $u$ is also bounded.
	
Now let $v_n=S_\lambda u_{\varphi_n}=S_\lambda^2u_n$. Then $v_n\in C(\overline{\Omega})\cap W(\Omega)$, $\lambda v_n-\cA v_n=0$ and $v_n(z)=\la\mu(z),u_{\varphi_n}\ra=:\psi_n(z)$ for all $n\in\CN $, $ z\in\partial\Omega$. Since $\partial \Omega$ is compact and $\mu$ is $\sigma (\mathscr{M}(\Omega), C_b(\Omega))$-continuous, it follows that the
set $\{ \mu (z): z\in\partial\Omega\}$ is $\sigma (\mathscr{M}(\Omega), C_b(\Omega))$-compact and thus, as a consequence of 
Prokhorov's theorem \cite[Theorem 8.6.2]{bogachev} \emph{tight}, i.e.\ given $\eps>0$ we find a compact set $K \subset \Omega$ such that
$\mu (z, \Omega\setminus K) \leq \eps$ for all $z \in \partial \Omega$. We should note that \cite[Theorem 8.6.2]{bogachev}
requires $\Omega$ to be a complete, separable metric space, whereas our $\Omega$ is not complete with respect to 
the usual metric. However,
as an open subset of the Polish space $\CR^d$, the set $\Omega$ itself is Polish, i.e.\ its topology is induced by a complete, separable
metric. It is with respect to this metric that we apply the theorem from \cite{bogachev}. 

As $u_n \to u$ locally uniformly, this clearly implies
that $\psi_n(z) \coloneqq \langle u_n,\mu(z)\rangle \to \psi(z)\coloneqq \langle u, \mu(z)\rangle$ uniformly in $z \in \partial \Omega$. In the case where $\Re\lambda>0$, it follows
from the complex maximum principle (Lemma \ref{l.maxprinciple2}) that $S_\lambda^2 u_n = u_{\psi_n} \to u_{\psi}$ uniformly on $\overline{\Omega}$. In the case where $\lambda =0$ we apply Alexandrov's maximum principle \cite[Theorem 9.1]{gt}
to real and imaginary parts separately to infer that $\Re u_{\psi_n} = u_{\Re \psi_n} \to u_{\Re\psi} = \Re u_{\psi}$
uniformly on $\overline{\Omega}$ and similarly for the imaginary parts. Also in this case it follows that $S_\lambda^2 u_n \to 
u_{\psi}$ uniformly on $\overline{\Omega}$. In both cases we have proved that $S_\lambda^2$ is compact.
\end{proof}

Lemma \ref{l.S-compact} will allow us to prove that $I -S_\lambda$ is invertible for $\Re\lambda>0$. In the proof we also use 
the following variant of the Fredholm alternative which is a special case of \cite[Theorem 15.4]{krein82} In an effort of being self-contained, we include a proof.

\begin{lem}\label{l.fredholm}
Let $X$ be a Banach space and $T \in \cL (X)$ be such that $T^2$ is compact. If $I-T$ is injective, then $I-T$ is surjective.
\end{lem}

\begin{proof}
Since $T^2$ is compact, $\sigma (T^2)\setminus\{0\}$ is countable with $0$ as only possible accumulation point. By the spectral mapping theorem, it follows that $\sigma (T) \setminus \{0\}$ is discrete. In particular every point in $\sigma (T)\setminus \{0\}$ is a boundary point, so that any point in that set belongs to the approximate point spectrum. 

Now assume that $I-T$ is injective but not surjective, so that $1 \in \sigma (T)\setminus \{0\}$. By the above there exists a sequence
$(x_n)$ with $\|x_n\| \equiv 1$ such that $(I-T)x_n \to 0$. As $T^2$ is compact, passing to a subsequence we may and shall assume that $T^2x_n$ converges, say to $y$. By continuity of $T$, we have $Tx_n - T^2x_n = T(I-T)x_n \to 0$ so that also $Tx_n \to y$.
It follows that $Ty= \lim T (Tx_n) = \lim T^2x_n =y$. As $I-T$ is injective, $y=0$. This contradicts the fact that $\|x_n\|\equiv 1$.
\end{proof}

\begin{lem}\label{l.S-invertible}
The operator $I-S_\lambda $ is invertible for $\Re\lambda>0$. If $A_{\mu}$ is injective, then also $I-S_0$ is invertible.
\end{lem}

\begin{proof}
By Lemma \ref{l.fredholm}, it suffices to prove that $I-S_\lambda$ is injective. However, if $u=S_\lambda u$ then 
$u \in D(A_\mu)$ and $\lambda u - \cA u =0$. If $\Re\lambda >0$ it follows from  Lemma \ref{l.amuinj} that $u=0$. 
In the case where $\lambda=0$, we have $A_\mu u =0$ and thus $u=0$ by assumption.
\end{proof}

In Lemma \ref{l.s-bounded} below we consider the complex Banach space $C(\overline{\Omega})$. For $v \in C(\overline{\Omega})$
we write $v\geq 0$ if $v$ is real valued and $v(x) \geq 0$ for all $x \in \overline{\Omega}$.

\begin{lem}\label{l.dominate}
Let $T_1, T_2: C(\overline{\Omega}) \to C(\overline{\Omega})$ be bounded linear operators. Assume that $T_2$
is positive (i.e.\ $T_2v \geq 0$ whenever $v\geq 0$) and that $|T_1v| \leq T_2v$ for all $v\geq 0$. Then
\[
|T_1v| \leq T_2|v|
\]
for all $v \in C(\overline{\Omega})$.
\end{lem}

\begin{proof}
The space $E= C(\overline{\Omega})$ is a complex Banach lattice. Thus the bidual $E''$ is an order complete Banach lattice and the evaluation map $j : E \to E''$ is an injective lattice homomorphism, see \cite[Theorem II.5.5]{schaefer} and Corollary 2 of that theorem. Now let $0\leq f \in E$ and $\theta \in \CR$. Then
\[
\Re (e^{i\theta}j(T_1f)) \leq |j (T_1f)| = j(|T_1 f|) \leq j(T_2 f).
\]
It follows from \cite[Theorem IV.1.8]{schaefer} and the two lines after \cite[Definition IV.1.7]{schaefer} that $|j(T_1v)| \leq j(T_2|v|)$ for all
$v \in E$. Hence $j(|T_1v|) = |j(T_1v)| \leq j(T_2|v|)$ which implies that $|T_1v| \leq T_2|v|$ as claimed.
\end{proof}

We now prove the estimate  in Proposition \ref{p.S-properties}. This finishes the proof  of Proposition \ref{p.S-properties}.

\begin{lem}\label{l.s-bounded}
	For all $\delta>0$ we have
	\[ 
		\sup_{\Re\lambda\geq\delta}\|(I-S_\lambda)^{-1}\|<\infty.
	\]
Moreover, $(I-S_\lambda)^{-1} \geq 0$ if $\lambda>0$ is real.
\end{lem}

\begin{proof}
	We first show that the function
	\[
		\CC_+\to\cL(C(\overline{\Omega})),\ \lambda\mapsto S_\lambda
	\]
	is holomorphic with $S_\lambda'=-R(\lambda, A_0)S_\lambda$. Here $\CC_+ \coloneqq \{ \lambda \in \CC : \Re \lambda >0\}$.
	
To this end, let $\lambda \in \CC_+$ and $ v\in C(\overline{\Omega})$ be given and, for small enough $\tau \in \CC_+$, 
set $u\coloneqq S_{\lambda+\tau}v$, $w\coloneqq S_\lambda v$. Then $u|_{\partial\Omega}=w|_{\partial\Omega}=\langle v,\mu \rangle$ and 
	\begin{align*}
		(\lambda+\tau)u-\cA u&=0\\
		\lambda w-\cA w&=0.
	\end{align*}
	It follows that $(u-w)|_{\partial\Omega}=0$ and $(\lambda+\tau)(u-w)-\cA (u-w)=-\tau w$. 
	Consequently $u-w=-R(\lambda+\tau, A_0)(\tau w)$.
	
	In other words
	\[
		\frac{1}{\tau}(S_{\lambda+\tau}v-S_\lambda v)=-R(\lambda+\tau, A_0)S_\lambda v.
	\]
	Sending $ \tau\to0 $ yields the claim.\medskip
	
It follows by induction that 
	\[ 
		\frac{d^n}{d\lambda^n}S_\lambda=(-1)^nn!R(\lambda, A_0)^nS_\lambda.
	\]
Noting that for $\lambda>0$ the operators $R(\lambda, A_0)$ and $S_\lambda$ are positive, it follows that
the function $ \lambda\mapsto S_\lambda $ on $ (0,\infty) $ is completely monotonic. 
Hence by Bernstein's Theorem (see \cite[Theorem 12a]{widder}) for each $0\leq v\in C(\overline{\Omega})$, $x\in\overline{\Omega}$ there exists an increasing function $\alpha\colon (0,\infty)\to\CR$ such that
	\[ 
		(S_\lambda v)(x)=\int_0^\infty e^{-\lambda t}  \, d\alpha(t)
	\]
for all $\lambda >0$. By the uniqueness theorem for holomorphic functions, the same formula also holds for complex $\lambda$ with $\Re\lambda >0$.	
	
	It follows that for $ \Re\lambda\geq\delta>0 $ we have
	\begin{equation}\label{eq.sest}
		|(S_\lambda v)(x)|\leq\int_0^\infty e^{-\delta t}\, d\alpha(t)=(S_\delta v)(x).
	\end{equation}
	
This proves that $S_\lambda$ is bounded on the half plane $\{\lambda\in\CC:\ \Re\lambda\geq\delta \}$.\medskip
	
	Since $S_\delta$ is a positive operator, the spectral radius $r(S_\delta)$ is in the spectrum of $S_\delta$, see \cite[Chapter V, Prop. 4.1]{schaefer}. Since $I-S_\delta$ is invertible by Lemma \ref{l.S-invertible} we have $r(S_\delta)<1$. Consequently,
	\[ 
		(I-S_\delta)^{-1}=\sum_{n=0}^\infty S^n_\delta
	\]
	where the sum converges absolutely in $\cL(C(\overline{\Omega}))$.  In particular, it follows that $(I-S_\delta)^{-1}\geq 0$.
	
Now let a complex $\lambda$ with $\Re\lambda \geq \delta$ be given. It follows inductively from \eqref{eq.sest} and Lemma
\ref{l.dominate} that
$|S^n_\lambda v| \leq S^n_\delta v$ whenever $v\geq 0$. Splitting a complex valued function in real and imaginary parts and those into positive and negative parts, it follows that $\|S^n_\lambda\| \leq 4 \|S^n_\delta\|$. This implies that also the series
	\[ 
	\sum_{n=0}^\infty S_\lambda^n 
	\]
converges, of course to $(I-S_\lambda)^{-1}$. It follows that 
	\[
		\|(I-S_\lambda)^{-1}\|\leq 4\|(I-S_\delta)^{-1}\|
	\]
for $ \Re\lambda\geq\delta $. 
\end{proof}

Having thus proved Proposition \ref{p.S-properties}, we are now ready to state and prove the main result of this section.

\begin{thm}\label{t.holom-semigroup}
The operator $A_{\mu}$, defined by \eqref{eq.amu}, generates a positive, holomorphic  semigroup 
$T_\mu$ on $L^\infty(\Omega)$. 
\end{thm}

\begin{proof}
Let $\lambda \in \CC$ with $\Re\lambda>0$ be given. We claim that $\lambda \in \rho (A_\mu)$ and
\begin{equation}\label{eq.amures}
R(\lambda, A_\mu) = (I-S_\lambda)^{-1}R(\lambda, A_0).
\end{equation}
In fact, let $f \in L^\infty(\Omega)$, $w= R(\lambda, A_0)f \in C_0(\Omega)$ and $v= (I-S_\lambda)^{-1}w$.  Then $S_\lambda v = v-w$. This implies $(\lambda - \cA)(v-w) = 0$, hence $(\lambda -\cA)v= (\lambda - \cA)w = f$. Moreover, by the definition of $S_\lambda$, we have 
\[
v(z) = v(z) - w(z) = \int_\Omega v(x)\mu (z, dx).
\]
This shows $v \in D(A_\mu)$ and $(\lambda - A_\mu)v=f$. We have proved that $\lambda - A_\mu$ is surjective. 
Since $\lambda - A_\mu$ is injective by Lemma \ref{l.amuinj}, it follows that $\lambda \in \rho (A_\mu)$ and that 
\eqref{eq.amures} holds. We add that \eqref{eq.amures} remains true for $\lambda=0$ if $I-S_0$ is invertible.\medskip

Since $A_0$ generates a bounded holomorphic semigroup, there exists $M_0$ such that $\|\lambda R(\lambda, A_0)\| \leq M_0$
whenever $\Re\lambda >0$. Taking Proposition \ref{p.S-properties} into account, it follows from \eqref{eq.amures} that
\[
\|\lambda R(\lambda, A_\mu)\| \leq M_0\sup_{\Re\lambda \geq \delta}\|(I-S_\lambda)^{-1}\| < \infty
\]
for all $\lambda$ with $\Re\lambda \geq \delta$. This is the desired holomorphic estimate and it follows from Theorem \ref{t.holest}
that $A_\mu$ generates a holomorphic semigroup $T_\mu$.

It follows from \eqref{eq.amures}, the positivity of $(I-S_\lambda)^{-1}$ and that of $R(\lambda, A_0)$ for $\lambda >0$ that
$R(\lambda, A_\mu)\geq 0$ for real $\lambda >0$. Now the positivity of $T_\mu$ follows from Proposition \ref{p.possg}.
\end{proof}	
	
Let us discuss the situation where $A_\mu$ is invertible in more detail. We have

\begin{cor}\label{c.asstable}
Assume additionally that $A_\mu$ is injective. Then $A_\mu$ generates a bounded holomorphic semigroup which is exponentially stable.
\end{cor}

\begin{proof}
It follows from Proposition \ref{p.S-properties} that $I-S_0$ is invertible. Now \eqref{eq.amures} implies that $0 \in \rho (A_\mu)$ and 
$R(0,A_\mu) = (I-S_0)^{-1}R(0, A_0)$. Consequently, $[0,\infty) \subset \rho (A_0)$ and Proposition \ref{p.posstable} implies that the semigroup generated by $A_\mu$ is exponentially stable.
\end{proof}

To prove further properties of the semigroup $T_\mu$ we use the following lemma, which shows in particular, that a continuous function satisfying our nonlocal boundary condition always attains a positive maximum in the interior of $\Omega$.

\begin{lem}\label{l.intmax}
Let $u \in C(\overline{\Omega})$ be real-valued such that $u(z) \leq \langle u, \mu (z)\rangle$ for all $z \in \partial \Omega$. If 
$c\coloneqq \max_{x \in \overline{\Omega}} u(x) >0$, then there exists $x_0 \in \Omega$ such that $u(x_0) =c$.
\end{lem}

\begin{proof}
Aiming for a contradiction, let us assume that $u(x)< c$ for all $x \in \Omega$. In this case, we have $u(z_0) = c$ for some $z_0 \in \partial \Omega$. Since $0<u(z_0) \leq \langle u, \mu (z_0)\rangle$, it follows that $\mu (z_0)\neq 0$. Pick $x_0$ in the support of $\mu (z_0)$. We find $r>0$ and $\delta >0$ such that $u(x) \leq c-\delta$ for all $x \in B(x,r)$ and $\eps \coloneqq\mu (z_0, B(x_0,r)) >0$.
The same computation as in the proof of Lemma \ref{l.amuinj} leads to a contradiction.
\end{proof}

As a first application of Lemma \ref{l.intmax}, we identify some situations in which $A_\mu$ is injective, so that Corollary \ref{c.asstable}
applies.

\begin{prop}\label{p.inv}
Assume that for every connected component $U$ of $\Omega$ on which $c_0$ vanishes we find a point $z_0\in \partial U$ with 
$\mu (z_0, \Omega)<1$.
Then $A_\mu$ is injective.
\end{prop}

\begin{proof}
Assume that $0\neq u \in D(A_\mu)$ satisfies $A_\mu u =0$. We may assume that $u^+\neq 0$, otherwise we replace $u$ with $-u$.
Then $\gamma \coloneqq \max_{x \in \overline{\Omega}} u(x)>0$. By Lemma \ref{l.intmax}, we find a point $x_0\in \Omega$ with 
$u(x_0) = \gamma$. Let $U$ be the connected component containing $x_0$. As a consequence of the strict maximum principle \cite[Theorem 9.6]{gt}, $u$ is constant on $U$, i.e.\ $u|_U = \gamma \one_U$. 

Since $\cA u = c_0 \gamma$ on $U$ it follows that $c_0$ vanishes on $U$. By our assumption, we find a point $z_0 \in \partial U$ with 
$\mu (z_0, \Omega) < 1$. But then we have that
\[
\gamma = u(z_0) = \langle u, \mu (z_0)\rangle \leq \langle u|_U, \mu (z_0)\rangle < \gamma.
\]
This is a contradiction. Consequently, we must have that $u=0$.
\end{proof}

We next prove that $T_\mu$ is a contraction semigroup.

\begin{prop}\label{p.tmucont}
We have $\|T_\mu(t)\| \leq 1$ for all $t>0$.
\end{prop}

\begin{proof}
As a consequence of Proposition \ref{p.contr} it suffices to prove $\|\lambda R(\lambda, A_\mu)\| \leq 1$ for all $\lambda >0$.
Taking the positivity of $R(\lambda, A_\mu)$ into account, we only need to show that $\lambda R(\lambda, A_\mu)\one \leq \one$.
Let us write $u = R(\lambda, A_\mu)\one$. By Lemma \ref{l.intmax} there exists a point $x_0 \in \Omega$ with $u(x_0) 
= c\coloneqq \max_{x \in \overline{\Omega}}u(x)$. As $\cA u = \lambda u - \one \in C(\overline{\Omega})$, we infer from 
Lemma \ref{l.localmax} that $\cA u(x_0) \leq 0$. Thus $\lambda u \leq \lambda u(x_0) = \lambda u(x_0) - 1 +1 = \cA u(x_0) +1 
\leq 1$. This finishes the proof.
\end{proof}

At this point, part (a) and the first half of part (b) of Theorem \ref{t.main} are proved. We end this section by proving that the semigroups generated by $A_\mu$ are monotonically increasing with respect to $\mu$.

\begin{prop}
Let $\mu_1, \mu_2 : \partial \Omega \to \mathscr{M}^+(\Omega)$ be two measure-valued functions satisfying Hypothesis \ref{hyp2} and denote the semigroups on $L^\infty(\Omega)$ generated by $A_{\mu_1}$ and $A_{\mu_2}$ by $T_{\mu_1}$ and $T_{\mu_2}$ respectively. If $\mu_1(z, A) \leq \mu_2(z,A)$ for all $z \in \partial \Omega$ and all Borel sets $A \subset \overline{\Omega}$, then
$0\leq T_{\mu_1}(t) \leq T_{\mu_2}(t)$ for all $t>0$.
\end{prop}

\begin{proof}
Let $\lambda >0$ and $0\leq f \in L^\infty(\Omega)$. We define $u_j \coloneqq R(\lambda, A_{\mu_j})f$ and $u \coloneqq u_1-u_2$.
It follows that $u \in C(\overline{\Omega})$ and
\begin{align*}
u(z) & = \langle u_1,\mu_1(z)\rangle - \langle u_2, \mu_2(z)\rangle = \langle u, \mu_1(z)\rangle - \langle u_2, \mu_2(z) - \mu_1(z)\rangle\\
& \leq \langle u, \mu_1(z)\rangle.
\end{align*}
As a consequence of Lemma \ref{l.intmax}, $u\leq 0$. As $f \geq 0$ was arbitrary, $R(\lambda, A_{\mu_1}) \leq R(\lambda, A_{\mu_2})$.
Now Proposition \ref{p.dom} yields $T_{\mu_1}(t) \leq T_{\mu_2}(t)$ for $t>0$ as claimed.
\end{proof}

\section{The strong Feller property and its consequences}

So far, we have considered the semigroup $T_\mu$ on the space $L^\infty(\Omega)$. In the theory of Markov processes, it is more natural to work on the space $B_b(\overline{\Omega})$ of all bounded and measurable functions on $\overline{\Omega}$ and to  consider so called \emph{kernel operators}. We write $K \coloneqq \overline{\Omega}$ and briefly recall the relevant notions in this situation. For this and further results on  kernel operators and semigroups of kernel operators (also in more general situations), 
we refer to \cite{k09, k11}.

A (bounded) \emph{kernel} on $K$ is a map $k: K\times \cB (K) \to \CC$ such that 
\begin{enumerate}[(i)]
\item the map $x \mapsto k(x,A)$ is Borel-measurable for all $A \in \cB (K)$, 
\item the map $A \mapsto k(x,A)$ is a (complex) measure on $\cB (K)$ for each $x \in K$ and
\item we have $\sup_{x\in E}|k|(x,K)<\infty$, where $|k|(x,\cdot)$ denotes the total variation of the measure $k(x,\cdot)$.
\end{enumerate}

Let $X= C(K)$ or $X= B_b(K)$. We call an operator $T \in \cL (X)$ a \emph{kernel operator} if there exists a kernel $k$ such that
\begin{equation}
\label{eq.rep}
Tf(x) = \int_K f(y)\, k(x, dy)
\end{equation}
for all $f \in X$ and $x \in K$. As there is at most one kernel $k$ satisfying \eqref{eq.rep}, we call $k$ the \emph{kernel associated with $T$} and, conversely, $T$ the \emph{operator associated with $k$}. Obviously, a kernel operator is positive if and only if the associated kernel consists of positive measures. 
\smallskip 

We note that every bounded operator on $C(K)$ is a kernel operator. Indeed, given $T \in \cL (C(K))$, we can set $k(x, \cdot ) \coloneqq T^*\delta_x \in \mathscr{M}(K)$. Using standard arguments (see e.g.\ the proof of \cite[Proposition 3.5]{k11}) one sees that 
$x \mapsto k(x,A)$ is measurable for any Borel set $A$. Thus, $k$ is a kernel. It is straightforward to see that the associated operator is $T$. Since $T$ is given by a kernel $k$, we can extend $T$ to  a kernel operator $\tilde{T}$ on $B_b(K)$ by defining $\tilde{T}f(x)$ by 
the right hand side of \eqref{eq.rep}. We call the operator $\tilde{T}$ the \emph{canonical extension} of $T$ to $B_b(K)$. We note that there might be other extensions of $T$ to a bounded operator on $B_b(K)$ but $\tilde{T}$ is the only one which is a kernel operator. 

A bounded operator on $B_b(K)$ need not be a kernel operator. It turns out that an operator $T \in \cL (B_b(K))$ is a kernel operator if and only if the adjoint $T^*$ leaves the space $\mathscr{M}(K)$ invariant. For us another characterization is more useful.

\begin{lem}\label{l.kernelchar}
Let $T \in \cL (B_b(K))$. The following are equivalent.
\begin{enumerate}
[(i)]
\item $T$ is a kernel operator.
\item $T$ is \emph{pointwise continuous}, i.e.\ if $f_n$ is a bounded sequence converging pointwise to $f$, then $Tf_n$ converges pointwise to $Tf$.
\end{enumerate}
If $T$ is positive, it suffices to consider bounded and increasing sequences in (ii).
\end{lem}

\begin{proof}
The implication ``(i) $\Rightarrow$ (ii)'' follows immediately from dominated convergence. For the converse, set 
$k(x,A) \coloneqq (T\one_A)(x)$. As $T$ operates on $B_b(K)$, the function $x \mapsto k(x,A)$ is measurable in $x$. Condition (ii) yields that $k(x,\cdot)$ is a measure, thus $k$ is a kernel. Using that simple functions are dense in $B_b(K)$, it is easy to see that $T$ is associated with $k$.
\end{proof}

\begin{cor}\label{c.kernelclosed}
The space of all kernel operators is norm closed in $\cL (B_b(K))$.
\end{cor}

Of particular interest are \emph{strong Feller operators}, i.e.\ kernel operators on $B_b(K)$ which only take values in $C(K)$. A bounded operator on $C(K)$ (which is automatically a kernel operator) is called \emph{strong Feller operator} if its canonical extension to $B_b(K)$ is strongly Feller.
It is easy to see that a kernel operator (on $C(K)$ or $B_b(K)$) is strongly Feller if and only if for the associated kernel $k$ the function $x \mapsto k(x,A)$ is continuous for every Borel set $A$. Using Corollary \ref{c.kernelclosed}, it follows that the set of all strong Feller operators is a norm closed subspace of $\cL (B_b(K))$.

The importance of strong Feller operators for us stems from the following result.

\begin{lem}\label{l.ultra}
Let $T,S$ be positive strong Feller operators. Then the product $ST$ is a compact operator on $\cL (B_b(K))$.
\end{lem}

\begin{proof}
It is well known that the product of two positive strong Feller operators is \emph{ultra Feller}, i.e.\ it maps bounded subsets of $B_b(K)$ to equicontinuous subsets of $C(K)$. A proof of this fact can be found in \cite[\S 1.5]{revuz}. As $K$ is compact, it follows from the Arzel\`a-Ascoli theorem that an equicontinuous subset of $C(K)$ is relatively compact. It follows that
an ultra Feller operator is compact.
\end{proof}

Let us now come back to the situation considered in Section \ref{sect.three}. We had operators $T \in \cL (L^\infty(\Omega))$ such that
$Tf \in C(\overline{\Omega}) = C(K)$ for all $f \in L^\infty(\Omega)$. In particular, we can consider the restriction $T_{C(K)}$ of such an operator to $C(K)$. By the above $T_{C(K)}$ is a kernel operator and thus has a canonical extension $\tilde{T}$ to $B_b(K)$. The obvious question is whether $\tilde{T} = T\circ \iota$ where $\iota: B_b(K) \to L^\infty(\Omega)$ maps a bounded measurable function to its equivalence class modulo equality almost everywhere. 

Unfortunately, this need not be the case. The problem is that $T\circ \iota$ need not be a kernel operator.
\begin{example}
We give an example for $K = \CN \cup \{\infty\}$ where the neighborhoods of $\infty$ are sets of the form $\{n, n+1, n+2, \ldots\}$. In this case $B_b(K) = \ell^\infty(K)$ and 
\[
C(K) = \{ (x_1, x_2, \ldots, x_\infty) : x_n \to x_\infty\}.
\]
Pick a Banach limit $\varphi$, i.e.\ a functional in $(\ell^\infty)^*$ with $\varphi (x) = \lim (x)$ for all convergent sequences $x$ (which is positive and satisfies $\varphi \circ L = \varphi$ for the left shift $L$).

Define $T \in \cL (\ell^\infty)$ by
\[
T(x_1, x_2, \ldots, x_\infty) = \varphi (x_1, x_2, x_3, \ldots) \cdot (1,1,1,\ldots, 1).
\]
Then $T$ indeed takes values in $C(K)$. The kernel associated with $T_{C(K)}$ is $k(x, A) = \delta_\infty (A)$. Thus the extension 
$\tilde{T}$ evaluates functions $x=(x_1, x_2, \ldots, x_\infty)$ at the point $\infty$. However, this need not be the value of $\varphi (x_1, x_2, \ldots)$, so that $\tilde T \neq T$.
\end{example}

Directly from the characterization in Lemma \ref{l.kernelchar} we obtain

\begin{lem}\label{l.linftysf}
Let $T \in \cL (L^\infty(\Omega))$ be a positive operator taking values in $C(\overline{\Omega})$ and let $\iota : B_b(\overline{\Omega}) \to L^\infty(\Omega)$ be as above. Then $T\circ \iota$ is a kernel operator if and only if whenever 
$f_n$ is a bounded, increasing sequence of positive functions in $L^\infty(\Omega)$ with $f(x) = \sup f_n(x)$ for almost every $x\in \Omega$, 
we have $Tf_n(x) \to Tf(x)$ for all $x \in \overline{\Omega}$.
In this case, $T\circ \iota$ is a strong Feller operator.
\end{lem}

Motivated by Lemma \ref{l.linftysf}, we call an operator $T$ on $L^\infty(\Omega)$ a \emph{strong Feller operator} if $TL^\infty(\Omega) \subset C(\overline{\Omega})$ and $T\circ \iota$ is a strong Feller operator. Let us note that the set of all strong Feller operators is a closed subspace of $\cL (L^\infty(\Omega))$. 
We also remark that if $A$ has Lebesgue measure zero, then $\iota (\one_A) =0$ and thus $T(\iota (\one_A)) = 0$ in $L^\infty(\Omega)$. This implies that
if $T \in \cL (L^\infty(\Omega))$ is strong Feller, then for the kernel $k$ associated with $T\circ \iota$
the measure $k(x, \cdot)$ is absolutely continuous with respect to Lebesgue measure on $\Omega$ for all $x \in \overline{\Omega}$.

\begin{cor}\label{l.linftyultra}
Let $T,S \in \cL (L^\infty(\Omega))$ be strong Feller operators. Then $TS$ is compact. 
\end{cor}

\begin{proof}
Let $f_n$ be a bounded sequence in $L^\infty(\Omega)$. We pick representatives $\tilde{f}_n$ of $f_n$ in $B_b(\Omega)$. Extending 
$\tilde{f}_n$ with zero to $\overline{\Omega}$, we find functions $\tilde{f}_n$ in $B_b(\overline{\Omega})$
with $\tilde f_n = f_n$ a.e.\ on $\Omega$. By Lemma \ref{l.ultra},
the operator $(S\circ \iota)(T\circ \iota)$ is compact. Since two continuous functions which are equal almost everywhere are equal, 
$(S\circ \iota)(T\circ \iota) = (ST)\circ \iota$. It follows from Lemma \ref{l.ultra} that $((ST)\circ \iota )\tilde{f}_n = STf_n$ has a (uniformly) converging subsequence.
\end{proof}

Concerning the operator $A_{\mu}$ we have

\begin{prop}\label{p.strongfeller}
The operator $R(\lambda, A_{\mu})$ is a strong Feller operator for each $\lambda >0$.
\end{prop}

\begin{proof}
Since $R(\lambda, A_\mu)$ is a positive operator on $L^\infty (\Omega)$ taking values in $C(\overline{\Omega})$ for $\lambda>0$,
it follows from Lemma \ref{l.linftysf}, that we only need to prove that 
$R(\lambda, A_{\mu})f_n \to R(\lambda, A_{\mu})f$ pointwise, whenever
$f_n$ is a bounded and increasing sequence converging pointwise almost everywhere to $f$.
It was seen in the proof of Theorem \ref{t.holom-semigroup} that $R(\lambda, A_{\mu})$ is positive, thus the sequence 
$u_n \coloneqq R(\lambda, A_{\mu})f_n$ is increasing. We set $u(x) \coloneqq \sup u_n(x)$ for all $x \in \overline{\Omega}$. 
It is a consequence of Proposition \ref{p.local-regularity} that $u_n$ converges locally uniformly in $\Omega$. Thus $u \in C_b(\Omega)$. We set 
$\varphi (z) \coloneqq \langle u, \mu (z)\rangle$. By the continuity assumption on $\mu$, the function $\varphi : \partial \Omega \to \CR$ is continuous. It follows from dominated convergence that
\begin{equation}\label{eq.urandbed}
\varphi (z) = \int_\Omega u(x) \, \mu (z, dx) = \lim_{n\to \infty} \int_\Omega u_n(x)\, \mu (z, dx) = \lim_{n\to\infty} u_n(z) = u(z).
\end{equation}
Here we have used that $u_n \in D(A_{\mu})$. As a consequence of Dini's theorem, $u_n$ converges uniformly to $\varphi$ on 
$\partial \Omega$.

We now consider the Poisson operator $\mathscr{P}$ on $L^d(\Omega) \oplus C(\partial \Omega)$ from the proof of 
Proposition \ref{p.wellposed}. Then we have $(u_n, 0) \in D(\mathscr{P})$ and $(\lambda - \mathscr{P}) (u_n,0) = (f_n, u_n|_{\partial\Omega})$. As $D(\mathscr{P}) \subset C(\overline{\Omega}) \oplus \{0\}$, it follows from the closed graph theorem that $R(\lambda ,\mathscr{P})$ is continuous as an operator from $L^d(\Omega)\oplus C(\partial \Omega)$ to $C(\overline{\Omega})\oplus \{0\}$, where 
$C(\overline{\Omega})$ is endowed with the topology of uniform convergence on $\overline{\Omega}$. By the above, 
$(f_n, u_n|_{\partial \Omega})$ converges to $(f, \varphi)$ in $L^d(\Omega) \oplus C(\partial \Omega)$. Thus 
\begin{equation}\label{eq.pois}
(u_n, 0) = R(\lambda, \mathscr{P})(f_n, u_n|_{\partial \Omega}) \to R(\lambda, \mathscr{P})(f, \varphi) =\colon (w, 0) 
\end{equation}
in $C(\overline{\Omega})\oplus \{0\}$. In particular, $u_n \to w$ uniformly on $\overline{\Omega}$. Since $u_n(x) \to u(x)$ for all $x \in \overline{\Omega}$, we must have $w(x)=u(x)$ for all $x \in \overline{\Omega}$. 
It follows from \eqref{eq.pois} that $w\in W^{2,d}_\mathrm{loc}(\Omega)\cap C(\overline{\Omega})$
and $\lambda w - \cA w = f$. Since $f\in L^\infty(\Omega)$, elliptic regularity (see Proposition \ref{p.wellposed}) implies $w \in W(\Omega)$. We have thus proved that $u=w \in C(\overline{\Omega})\cap W(\Omega)$. In view of \eqref{eq.urandbed}, it follows that $u \in D(A_\mu)$ and $\lambda u - A_\mu u = f$.

This shows that $R(\lambda, A_\mu )f_n$ converges pointwise to $R(\lambda, A_\mu) f$, proving that $R(\lambda, A_\mu)$ is a strong Feller operator.
\end{proof}

We now obtain more information about the resolvent $R(\lambda, A_{\mu})$ and the semigroup $T_\mu$ generated by $A_{\mu}$. 

\begin{cor}\label{c.strongfeller}
For every $\lambda \in \rho (A_\mu)$ the operator $R(\lambda, A_{\mu})$ is strongly Feller and compact; moreover also $T_\mu(t)$ is strongly Feller and compact for $t>0$, where $T_\mu$ is the semigroup generated by $A_{\mu}$.
\end{cor}

\begin{proof}
The function $\rho (A_\mu) \ni \lambda \mapsto R(\lambda, A_{\mu})$ is analytic with values in 
$\cL (L^\infty(\Omega))$. On $(0,\infty)$, it takes values in the closed subspace of strong Feller operators as a consequence of Proposition \ref{p.strongfeller}. 
Since $\rho (A_\mu)$ contains $\Sigma_\theta \setminus \{0\}$ for a suitable $\theta >\frac{\pi}{2}$ and the latter set is
connected, the uniqueness theorem for holomorphic functions \cite[Proposition A.2]{abhn} implies that $R(\lambda, A_{\mu})$ is a strong Feller operator for every $\lambda \in \Sigma_\theta\setminus \{0\}$.
As the semigroup $T_\mu$ can be computed from the resolvent via a (operator-valued) Bochner integral over a contour in
$\Sigma_\theta\setminus\{0\}$, it follows that the semigroup $T_\mu$ consists of strong Feller operators, too.

Now Lemma \ref{l.ultra} implies that $T_\mu (t) = T_\mu(t/2)T_\mu(t/2)$ is compact for all $t>0$. Consequently, also the resolvent
$R(\lambda, A_\mu)$, being given as a Bochner integral $R(\lambda, A_\mu ) = \int_0^\infty e^{-\lambda t}T_\mu (t)\, dt$ for $\Re \lambda>0$, consists of compact operators. In particular, $\sigma (A_\mu)$ is countable and thus $\rho (A_\mu)$ is connected.
Thus, invoking the uniqueness theorem for holomorphic functions a second time, it follows that $R(\lambda, A_\mu)$ is a strong Feller operator for all $\lambda \in \rho (A_\mu)$.
\end{proof}

In particular, Corollary \ref{c.strongfeller} yields the rest of part (b) and part (c) of Theorem \ref{t.main}.
We now finish the proof of Theorem \ref{t.main}.

\begin{proof}[Proof of part (d) of Theorem \ref{t.main}]
If $A_\mu$ is invertible, then $T_\mu$ is exponentially stable by Corollary \ref{c.asstable}. Thus in this situation, assertion (d) 
of Theorem \ref{t.main} is valid for $P=0$. 

So let us now assume that $0 \in \sigma (A_\mu)$. Since $A_\mu$ has compact resolvent and since $\|\lambda R(\lambda, A)\| \leq 1$
for $ \lambda >0$, 
it follows that $A_\mu$ has a pole of order 1 at 0. Since $\omega + \Sigma_\theta \subset \rho (A_\mu)$ for some
$\theta \in (\frac{\pi}{2},\pi)$ and some $\omega \geq 0$ (see the proof of Proposition \ref{p.expstable}), it follows that 
$\sigma (A_\mu)\cap i\CR$ is finite. On the other hand, \cite[Remark C-III.2.15]{schreibmaschine} shows that
$\sigma (A_\mu)\cap i\CR$ is cyclic, i.e.\ if $is \in \sigma (A_\mu)$ for some $s \in \CR$, then also $iks \in \sigma (A)$
for all $k \in \CZ$. Consequently,
$\sigma (A_\mu)\cap i\CR =\{0\}$. This implies that $\{0\}$ is a dominating eigenvalue, i.e.\ there exists $\eps>0$ such that
$\sigma (A_\mu)\setminus \{0\} \subset \{ \lambda \in \CC : \Re \lambda \leq -\eps\}$. 

Denote by $P$ the residuum of $R(\lambda, A_\mu)$ at $0$; this is the same as the spectral projection associated with $\{0\}$. 
It follows from the representation of $T_\mu$ as a contour integral that for suitable constants $M, \eps>0$ we have
 $\|T(t)(I-P)\| \leq Me^{-\eps t}$ for all
$t>0$, see \cite[Theorem 2.6.2]{abhn}. Since on the other hand $T(t)P = P$ for $t>0$, we find
\[
\|T(t)-P\| = \|T(t)(I-P) + T(t)P - P\| = \|T(t)(I-P)\| \leq Me^{-\eps t}
\]
for all $t>0$.
\end{proof}

\begin{rem}
Part (d) of Theorem \ref{t.main} implies that the operator $P$ is always of finite rank. This might be surprising since $\Omega$ may have infinitely many connected components. Let us illustrate how Hypothesis \ref{hyp2} is responsible for the behavior in Theorem \ref{t.main}(d).

Let us assume that $\Omega$ consists of a countable number of connected components $(\Omega_k)_{k\geq 1}$ such that also the closures $\overline{\Omega_k}$ are pairwise disjoint. Such an open set $\Omega$ can be Dirichlet regular, e.g.\ in dimension one where \emph{every} bounded open set is Dirichlet regular.

Let us moreover assume that $\mu$ consists of probability measures and that via $\mu$ there is ``no communication'' between the connected components, more precisely $\mu (z, \Omega_k)=1$ for all $z \in \partial \Omega_k$ and all $k\in\CN$.  
In this situation one would expect the kernel of $A_\mu$ be infinite, thus $P$ be not of finite rank. However, it turns out that in this situation Hypothesis \ref{hyp2} does not hold.

To see this, pick $z_n \in \partial \Omega_n$. It follows from the boundedness of $\Omega$, that this sequence has an accumulation point $z_0 \in \partial \Omega$, say the subsequence $z_{n_k}$ converges to $z_0$. Noting that $\one_{\Omega_j}$ is a continuous function on $\Omega$, it would follow from the $\sigma (\mathscr{M}(\Omega), C_b(\Omega))$-continuity of $z \mapsto \mu(z)$, that  
\[
\mu(z_{n_k}, \Omega_j) =\int \one_{\Omega_j}\, d\mu(z_{n_k}) \to \int \one_{\Omega_j}\, d\mu(z_0) =\mu (z_0, \Omega_j)
\]
 as $k \to \infty$ for all $j\in \CN$. As $\mu(z_{n_k}, \Omega_j) = 0$ for $k$ large enough,
 $\mu (z_0, \Omega_j) = 0$ for all $j\in \CN$. But then we would have  that $\mu (z_0, \Omega) = \sum_{n\in \CN} \mu (z, \Omega_n) =0$. This contradicts the assumption that every measure $\mu(z)$ is a probability measure on $\Omega$.
\end{rem}

We now describe the asymptotic behavior of the semigroup $T_\mu$ in  the case where $\Omega$ is connected.

\begin{cor}\label{c.asympt}
Assume that $\Omega$ is connected. 
\begin{enumerate}
[(a)]
\item  If $c_0 \neq 0$ or there exists a point $z\in \partial \Omega$ with $\mu (z, \Omega) < 1$, then 
$P=0$, i.e.\ $T_\mu$ is exponentially stable.
\item If $c_0= 0$ and $\mu(z)$ is a probability measure on $\Omega$ for every $z \in \partial \Omega$, then there exists a function $0\leq h \in L^1(\Omega)$ with $\int_\Omega h\, dx =1$ such that
\[
P f = \int_\Omega f h\, dx \cdot \one_{\overline{\Omega}}.
\]
\end{enumerate}
\end{cor}

\begin{proof}
(a) By Proposition  \ref{p.inv} $A_\mu$ is injective. Thus
$T_\mu$ is exponentially stable by Corollary \ref{c.asstable}.\smallskip

(b) In this case, $\one_{\overline{\Omega}} \in \ker A_\mu$. It follows from the proof of Proposition \ref{p.inv}, that $\ker A_\mu = \CC \cdot \one_{\overline{\Omega}}$. Thus $P$ is a rank one projection, i.e.\ $Pf = \varphi (f) \cdot \one_{\overline{\Omega}}$. 
Since $T_\mu (t) \to P$ in operator norm, it follows that $P$ is a strong Feller operator. In particular, if $f_n$ is a bounded increasing
sequence converging a.e.\ to $f$ then $Pf_n \to Pf$ pointwise, i.e.\ $\varphi (f_n) \to \varphi (f)$. But this implies that $\varphi (f) = 
\int f h\, dx$ for some $0\leq h \in L^1(\Omega)$. Indeed, if we set $\nu (A) = \varphi ([\one_A])$, where $[f]$ denotes the equivalence class of $f$ modulo equality almost everywhere, then it follows from the additional 
continuity property of $\varphi$ that $\nu$ is a measure.
Obviously $\nu$ is absolutely continuous with respect to Lebesgue measure whence it has a density $h$ by the Radon--Nikodym theorem.
\end{proof}

\begin{rem}
Let us comment on the speed of convergence, i.e.\ the possible choices for 
$\eps$ in Theorem \ref{t.main}(d). Since our semigroup is analytic,
the growth bound and the spectral bound coincide, see \cite[Corollary 2.3.2]{lunardi}. In case (a) of the above corollary where $P=0$ the spectral bound $s(A_\mu)<0$ belongs to the spectrum (since the semigroup is positive),
 and we can allow every $\eps$ with $s(A_\mu) < - \eps < 0$ in \ref{t.main}(d). In case (b), let
$s_* \coloneqq \sup \{ \Re \lambda : \lambda \in \sigma (A_\mu) \setminus \{0\} \}$. Then we can allow every $\eps$
with $s_* < -\eps < 0$.

\end{rem}

\section*{Acknowledgement}
The authors thank Zakhar Kabluchko for some valuable discussions and, in particular, for pointing out references \cite{b-ap07, b-ap09}
to us.

\def\cprime{$'$}


\begin{thebibliography}{10}

\bibitem{abhn}
{\sc W.~Arendt, C.~J.~K. Batty, M.~Hieber, and F.~Neubrander}, {\em
  Vector-valued {L}aplace Transforms and {C}auchy Problems}, vol.~96 of
  Monographs in Mathematics, Birkh\"auser Verlag, Basel, 2011.

\bibitem{ab99}
{\sc W.~Arendt and P.~B{\'e}nilan}, {\em Wiener regularity and heat semigroups
  on spaces of continuous functions}, in Topics in nonlinear analysis, vol.~35
  of Progr. Nonlinear Differential Equations Appl., Birkh\"auser, Basel, 1999,
  pp.~29--49.

\bibitem{ad08}
{\sc W.~Arendt and D.~Daners}, {\em Varying domains: stability of the
  {D}irichlet and the {P}oisson problem}, Discrete Contin. Dyn. Syst., 21
  (2008), pp.~21--39.

\bibitem{schreibmaschine}
{\sc W.~Arendt, A.~Grabosch, G.~Greiner, U.~Groh, H.~P. Lotz, U.~Moustakas,
  R.~Nagel, F.~Neubrander, and U.~Schlotterbeck}, {\em One-parameter semigroups
  of positive operators}, vol.~1184 of Lecture Notes in Mathematics,
  Springer-Verlag, Berlin, 1986.

\bibitem{as14}
{\sc W.~Arendt and R.~Sch\"atzle}, {\em Semigroups generated by elliptic
  operators in non-divergence form on ${C_0(\Omega)}$}.
\newblock To appear in Ann.\ Sc.\ Norm.\ Super.\ Pisa Cl.\ Sci., XIII, 2014.

\bibitem{b-ap07}
{\sc I.~Ben-Ari and R.~G. Pinsky}, {\em Spectral analysis of a family of
  second-order elliptic operators with nonlocal boundary condition indexed by a
  probability measure}, J. Funct. Anal., 251 (2007), pp.~122--140.

\bibitem{b-ap09}
\leavevmode\vrule height 2pt depth -1.6pt width 23pt, {\em Ergodic behavior of
  diffusions with random jumps from the boundary}, Stochastic Process. Appl.,
  119 (2009), pp.~864--881.

\bibitem{bogachev}
{\sc V.~I. Bogachev}, {\em Measure theory. {V}ol. {I}, {II}}, Springer-Verlag,
  Berlin, 2007.

\bibitem{con}
{\sc J.~B. Conway}, {\em Functions of one complex variable. {II}}, vol.~159 of
  Graduate Texts in Mathematics, Springer-Verlag, New York, 1995.

\bibitem{dl}
{\sc R.~Dautray and J.-L. Lions}, {\em Mathematical analysis and numerical
  methods for science and technology. {V}ol. 1}, Springer-Verlag, Berlin, 1990.
\newblock Physical origins and classical methods, With the collaboration of
  Philippe B{\'e}nilan, Michel Cessenat, Andr{\'e} Gervat, Alain Kavenoky and
  H{\'e}l{\`e}ne Lanchon, Translated from the French by Ian N. Sneddon, With a
  preface by Jean Teillac.

\bibitem{day}
{\sc W.~A. Day}, {\em A decreasing property of solutions of parabolic equations
  with applications to thermoelasticity}, Quart. Appl. Math., 40 (1982/83),
  pp.~468--475.

\bibitem{ek}
{\sc S.~N. Ethier and T.~G. Kurtz}, {\em Markov processes}, Wiley Series in
  Probability and Mathematical Statistics: Probability and Mathematical
  Statistics, John Wiley \& Sons, Inc., New York, 1986.
\newblock Characterization and convergence.

\bibitem{feller-semigroup}
{\sc W.~Feller}, {\em The parabolic differential equations and the associated
  semi-groups of transformations}, Ann. of Math. (2), 55 (1952), pp.~468--519.

\bibitem{feller-diffusion}
\leavevmode\vrule height 2pt depth -1.6pt width 23pt, {\em Diffusion processes
  in one dimension}, Trans. Amer. Math. Soc., 77 (1954), pp.~1--31.

\bibitem{galaskub}
{\sc E.~I. Galakhov and A.~L. Skubachevski{\u\i}}, {\em On {F}eller semigroups
  generated by elliptic operators with integro-differential boundary
  conditions}, J. Differential Equations, 176 (2001), pp.~315--355.

\bibitem{gt}
{\sc D.~Gilbarg and N.~Trudinger}, {\em Elliptic Partial Differential Equations
  of Second Order}, Springer-Verlag, Berlin, 2001.

\bibitem{gk02}
{\sc I.~Grigorescu and M.~Kang}, {\em Brownian motion on the figure eight}, J.
  Theoret. Probab., 15 (2002), pp.~817--844.

\bibitem{gk07}
\leavevmode\vrule height 2pt depth -1.6pt width 23pt, {\em Ergodic properties
  of multidimensional {B}rownian motion with rebirth}, Electron. J. Probab., 12
  (2007), pp.~no. 48, 1299--1322.

\bibitem{gjs}
{\sc P.~Gurevich, W.~J{\"a}ger, and A.~Skubachevskii}, {\em On periodicity of
  solutions for thermocontrol problems with hysteresis-type switches}, SIAM J.
  Math. Anal., 41 (2009), pp.~733--752.

\bibitem{gur08}
{\sc P.~L. Gurevich}, {\em Bounded perturbations of two-dimensional diffusion
  processes with nonlocal conditions near the boundary}, Mat. Zametki, 83
  (2008), pp.~181--198.

\bibitem{krein82}
{\sc S.~G. Kre{\u\i}n}, {\em Linear equations in {B}anach spaces},
  Birkh\"auser, Boston, Mass., 1982.
\newblock Translated from the Russian by A. Iacob, With an introduction by I.
  Gohberg.

\bibitem{k67}
{\sc N.~V. Krylov}, {\em The first boundary value problem for elliptic
  equations of second order}, Differencial\cprime nye Uravnenija, 3 (1967),
  pp.~315--326.

\bibitem{k09}
{\sc M.~Kunze}, {\em Continuity and equicontinuity of semigroups on norming
  dualpairs}, Semigroup Forum, 79 (2009), pp.~540--560.

\bibitem{k11}
\leavevmode\vrule height 2pt depth -1.6pt width 23pt, {\em A {P}ettis-type
  integral and applications to transition semigroups}, Czechoslovak Math. J.,
  61 (2011), pp.~437--459.

\bibitem{lunardi}
{\sc A.~Lunardi}, {\em Analytic semigroups and optimal regularity in parabolic
  problems}, Modern Birkh\"auser Classics, Birkh\"auser/Springer Basel AG,
  Basel, 1995.

\bibitem{pl13}
{\sc J.~Peng and W.~V. Li}, {\em Diffusions with holding and jumping boundary},
  Sci. China Math., 56 (2013), pp.~161--176.

\bibitem{revuz}
{\sc D.~Revuz}, {\em Markov chains}, North-Holland Publishing Co., Amsterdam,
  1975.
\newblock North-Holland Mathematical Library, Vol. 11.

\bibitem{schaefer}
{\sc H.~Schaefer}, {\em Banach lattice and positive Operators}, Springer, 1974.

\bibitem{t88}
{\sc K.~Taira}, {\em Diffusion processes and partial differential equations},
  Academic Press, Inc., Boston, MA, 1988.

\bibitem{taira}
\leavevmode\vrule height 2pt depth -1.6pt width 23pt, {\em Semigroups, boundary
  value problems and {M}arkov processes}, Springer Monographs in Mathematics,
  Springer-Verlag, Berlin, 2004.

\bibitem{wenzell}
{\sc A.~D. Ventcel$'$}, {\em On boundary conditions for multi-dimensional
  diffusion processes}, Theor. Probability Appl., 4 (1959), pp.~164--177.

\bibitem{widder}
{\sc D.~Widder}, {\em The {L}aplace {T}ransform}, Princeton Mathematical
  Series, v. 6, Princeton University Press, Princeton, N. J., 1941.

\end{thebibliography}
\end{document}